\documentclass[reqno, 10pt]{amsart}

\usepackage{amsmath, amssymb,epic,graphicx,mathrsfs,enumerate}

\usepackage{mathtools}
\usepackage{amsfonts}
\usepackage{amsthm}
\usepackage{latexsym}
\usepackage{longtable}
\usepackage{epsfig}
\usepackage{amsmath}
\usepackage{amssymb}
\newtheorem{thm}{Theorem}

\newtheorem{lem}[thm]{Lemma}
\newtheorem{pro}[thm]{Proposition}
\newtheorem{cor}[thm]{Corollary}

\theoremstyle{definition}

\DeclareMathOperator{\inn}{Inn} \DeclareMathOperator{\perm}{Sym}
 \DeclareMathOperator{\soc}{soc}
\DeclareMathOperator{\aut}{Aut} \DeclareMathOperator{\out}{Out}
 \DeclareMathOperator{\M}{M}

 \DeclareMathOperator{\frat}{Frat}
\DeclareMathOperator{\ssl}{SL}

\DeclareMathOperator{\sym}{Sym}

\DeclareMathOperator{\GL}{GL}

\DeclareMathOperator{\alt}{Alt}

\renewcommand{\emptyset}{\varnothing}

\title[Normal coverings]{Covers and normal covers of finite groups}
\author{Martino Garonzi}
\address{Dipartimento di Matematica,
Via Trieste 63, 35121 Padova, Italy.}
\email{mgaronzi@gmail.com}
\author{Andrea Lucchini}
\address{Dipartimento di Matematica,
Via Trieste 63, 35121 Padova, Italy.}
\email{lucchini@math.unipd.it}

\thanks{Research partially supported
by MIUR-Italy via PRIN Group theory and applications}


\begin{document}
\begin{abstract}
For a finite non cyclic group $G$, let $\gamma(G)$ be the smallest integer $k$ such that
$G$ contains $k$ proper subgroups $H_1,\dots,H_k$ with the property that every element of $G$ is contained in $H_i^g$ for
some $i \in \{1,\dots,k\}$ and $g \in G.$ We prove that if $G$ is a noncyclic permutation group of degree $n,$ then
$\gamma(G)\leq (n+2)/2.$ We then investigate the structure of the groups $G$ with $\gamma(G)=\sigma(G)$
(where $\sigma(G)$ is the size of a minimal cover of $G$) and
of those with $\gamma(G)=2.$
\end{abstract}
\maketitle

\section{Introduction}

Let $G$ be a non-cyclic finite group. A collection $\mathcal C$ of proper subgroups of $G$ is a cover
of $G$ if $\cup_{H\in \mathcal C}H=G;$  it is a minimal cover if $|\mathcal C|$ is as small as possible. A normal cover
has the property that $H^g \in \mathcal C$ for all $H \in \mathcal C,$ $g \in G.$ The covering number of $G,$ denoted
$\sigma(G),$ is the size of a minimal cover, and the normal covering number, denoted
$\gamma(G),$ is the smallest number of conjugacy classes of subgroups in a normal cover of $G.$
If $G$ is cyclic we pose $\sigma(G)=\gamma(G)=\infty,$ with the convention that $n<\infty$ for every integer $n.$

The first question on finite covers  was posed
by Scorza in 1926 \cite{sc} who settled the question which groups are the union of three proper
subgroups.  Cohn's 1994 paper \cite{cohn} brought Scorza's original question again to the forefront of research
in group theory and got the attention of many researchers (see for example  \cite{Bl}, \cite{BEGHM}, \cite{BFS}, \cite{cohn}, \cite{cubo}, \cite{mg}, \cite{gl} \cite{gm}, \cite{attila}, \cite{tom}).

The study of normal covers is an off-shoot of the finite covering problem
and relatively new (\cite{bbh}, \cite{mb},  \cite{blw}, \cite{bp}).
The first available results seem to indicate that the arguments used to investigate $\sigma(G)$ fail
when applied to the study of $\gamma(G)$ and this second invariant seems more difficult to be estimated.
For example, by the main result in Tomkinson's paper  \cite[Theorem 2.2.]{tom}, if $G$ is a finite soluble group then
$\sigma(G)=|W|+1,$ where $W$ is a chief factor of $G$ with least order among chief factors of $G$ with multiple complements;
in particular $\sigma(G)-1$ is a prime power. A similar formula for $\gamma(G)$ when $G$ is soluble is missing and in any case
$\gamma(G)$ has a surprisingly different behavior: for every $n\geq 2$, there exists a finite soluble group $G$ with $\gamma(G)=n$
\cite{cl}.

In this paper we address two questions related to the behavior of $\gamma(G)$.
We study the groups $G$ with $\sigma(G)=\gamma(G)$ and those with $\gamma(G)=2.$

In order to deal with the first question we start recalling  a lower bound for $\sigma(G),$ proved by Cohn.
Let $\mu(G)$ be the least integer $k$  such that $G$ has more than one maximal subgroup of index $k.$
Then we have:
\begin{pro}[Cohn \cite{cohn}, Corollary after Lemma 8] If $G$ is a finite group, then $\sigma(G)\geq \mu(G)+1.$
\end{pro}
On the other hand it turns out that the same value $\mu(G)+1$ represents an upper bound for $\gamma(G)$.
Indeed we prove:
\begin{pro}\label{basso}
If $G$ is a finite group, then $\gamma(G)\leq \mu(G)+1.$ Moreover $\gamma(G)=\mu(G)+1$ if and only
if $\mu(G)$ is a prime, $G$ contains at least two normal subgroups of index $\mu(G)$ and $\gamma(G)=\gamma(G/G^\prime).$
\end{pro}

\begin{cor}Suppose that $G$ is a noncyclic finite group. If $\sigma(G)=\gamma(G),$ then $p=\sigma(G)-1$ is a prime
and $G$ has a minimal cover consisting of normal subgroups of index $p$. In particular
$\gamma(G)=\gamma(G/G^\prime)=\sigma(G/G^\prime).$
\end{cor}

Proposition \ref{basso} is a consequence of a more general result, bounding $\gamma(G)$ when $G$ is
a noncyclic permutation group.

\begin{thm}\label{permut}If  $G$ is a noncyclic  permutation group of degree $n,$
then $\gamma(G) \leq (n+2)/2.$
\end{thm}

We may complete the previous statement noticing that the upper bound is reached infinitely often: if $p$ is any prime and $G$ is a subgroup of $\sym(2p)$ generated by two disjoint $p$-cycles then $G \cong C_p \times C_p$ so $\gamma(G) = p+1 = (2p+2)/2$.

\

It is interesting to study the groups for which $\sigma$ or $\gamma$ takes
the smallest possible value. No finite group can be expressed as a union of two proper subgroups or as a union of conjugates of a proper subgroup;
so $\sigma(G)\geq 3$ and $\gamma(G)\geq 2.$ Scorza's Theorem says that $\sigma(G)=3$ if and only if $G$ is the union of three
subgroups of index 2; this is equivalent to say that if $\sigma(G)=3$ but $\sigma(G/N)>3$ for every nontrivial normal subgroup $N$ of $G,$
then $G\cong C_2\times C_2.$ One could expect that, in a similar way, there are only few groups $G$ such that $\gamma(G)=2$ but $\gamma(G/N)>2$ for every nontrivial normal subgroup $N$ of $G,$ however it is not precisely like that. Indeed we will give many different examples of groups
$G$ with $\gamma(G)=2.$ However some restrictions on the structure of these groups can be proved.

\begin{thm}Assume that $\gamma(G)=2$ but $\gamma(G/N)>2$ for every nontrivial normal subgroup $N$ of $G.$
Then $G$ has a unique minimal normal subgroup $N$. Moreover if $G$ is covered with the conjugates of two maximal subgroups
then either one of these two subgroups contains $\soc(G)$ or $G$ is an almost simple group.
\end{thm}

On the other hand, as  we will recall in Section \ref{gamma2},  there are several different examples
of almost simple groups $G$ with $\gamma(G)=2$. Moreover in the same section we will construct infinite families
of examples of groups $G$ with a unique minimal normal subgroup $N$, covered by the conjugates of two maximal subgroups
$H$ and $K$, in which $H$ contains $N$ but the intersection of $K$ with $N$ has different behaviors:
trivial (when $N$ is abelian), of diagonal type, of product type. The conclusion is that there are several
different ways in which a finite group can be covered by the conjugates of two proper subgroups and a complete classification is quite difficult.

\subsection*{Acknoledgements} We would like to thank Attila A. Mar{\'o}ti and Pablo Spiga for fruitful discussions and  valuable and helpful  comments.

\section{Groups $G$ with $\gamma(G)=\sigma(G)$}\label{sdue}

We start this section with some preliminary results.

\begin{lem}\label{solva}
Let $G$ be a finite soluble noncyclic group such that $G/G'$ is cyclic. Then $\gamma(G) = 2$.
\end{lem}

\begin{proof}  We make induction on the order of $G.$
Since $\gamma(G) \leq \gamma(G/N)$ for every normal subgroup $N$ of $G$, we may assume that every proper quotient of $G$ is cyclic.
Together with the fact that $G/G'$ is cyclic, this implies that $G$ contains a unique minimal normal subgroup, say $N$, and $N$ has
a cyclic complement $M$. Moreover $M$ has precisely $|N|$ conjugates in $G$.  Let $K$ be a conjugate of $M$ in $G$, with $K\neq M.$ Since $M$ is cyclic, so is $K$ and $K \cap M \unlhd \langle K,M \rangle = G$. Since $K \cap M \not \supseteq N$ it follows $K \cap M = 1$. The $|N|$ conjugates of $M$ together with $N$ cover in total $$|N| + (|G:N|-1)|N| = |G|$$elements of $G$. It follows that $\gamma(G) = 2$.
\end{proof}




Denote by $m(G)$ the smallest index of a proper subgroup of $G.$ The following consequence
of the classification of the finite simple groups plays a crucial role in our proof.

\begin{pro} \label{as}
Let $X$ be an almost simple group. If $X\neq \aut (\alt(6))$ then $\gamma(X) < m(\soc(X))/2$.
Moreover $\gamma(\aut (\alt(6))=3.$
\end{pro}

\begin{proof}
Let $S = \text{soc}(X)$. For the value of $m(S)$ we refer to \cite[Table 5.2.A]{kl} and \cite[Table 1]{lucdam}.

If $S$ is an alternating group of degree $n \geq 5$ then $\gamma(X) < n/2= m(S)/2$ \cite{bp} unless $n=6.$
Moreover it is easy to check using \cite{gap} that if $S=\alt(6)$ then $\gamma(X)\leq 3$ with equality only if $X=\aut (\alt(6)).$


Suppose that $S$ is a sporadic simple group. It can be deduced from  \cite[Table 1]{marspo} that $\gamma(M_{11})=2$, $\gamma(M_{12})\leq 3,$
$\gamma(S)\leq 9$ if $S$ is not the Monster group $\M$ and $\gamma(\M) \leq 14:$ this is sufficient to conclude $\gamma(S)<m(S)/2.$ If $X$ is not simple then $X/S \cong C_2$ and $X$ has at most six conjugacy classes of involutions, and precisely 3 conjugacy classes of involutions if $X=\aut(M_{12}).$ Since every element of $X$ of odd order lies in $S$ and every element of $X$ of even order centralizes an involution, $\gamma(X) \leq 6+1 = 7  < m(S)/2$ if $S\neq M_{12},$   $\gamma(\aut(M_{12}))\leq 1+3=4.$

Suppose that $S$ is a simple group of Lie type. Denote by $q = p^f$ the size of the base field $F$, where $p$ is the characteristic.
Since $X$ is the union of the centralizers of the nontrivial elements of $S$ \cite[Proposition 7]{cubo}, $\gamma(X)\leq k^*(S)$, the number of conjugacy classes of elements in $S$ of prime order. In the case $S \neq A_m(q)$ we will prove that $k^*(S) < k(S) \leq m(S)/2,$
by using the bounds for the number $k(S)$ of conjugacy classes in $S$ proved in \cite[Corollary 1.2 and Tables 1 and 2]{chev}. Suppose that $S$ is of classical type and let $n$ be the dimension of the natural module over $F$. In \cite{ash}, eight collections $\mathcal{C}_1, \ldots, \mathcal{C}_8$ of natural subgroups of $X$ are defined, and each cyclic subgroup of $X$ is contained in one of these subgroups. So
$X$ is covered by the maximal subgroups of $X$ belonging to these Aschbacher classes. In the particular case when $S=A_m(q)$, we have $n=m+1,$
$S\cong PSL(n+1,q)$ and the number of conjugacy classes of subgroups of type $\mathcal{C}_1, \ldots, \mathcal{C}_8$ is at most
$2\cdot n+3\cdot d(n)+\log n + \log f +5 \leq 5(n+1)+\log n + \log q$
where $\log = \log_2$, $d(n)$ is the number of divisors of $n$ and $\omega(f)$ is the number of prime divisors of $f$ \cite[p. 69]{kalu}.
In the case $S=A_m(q)$ we will prove that
\begin{equation} \label{c1c8}
5(n+1)+\log n +\log q < m(S)/2
\end{equation}
with finitely many exceptions. We are now ready to start our case by case analysis.


\begin{itemize}

\item $S = A_m(q)$, $n = m+1$, $m \geq 1$, $(n,q) \neq (2,2), (2,3)$.
We have $m(S) = \frac{q^n-1}{q-1}$ if $(n,q) \neq (2,5), (2,7), (2,9), (2,11), (4,2)$, $m(A_1(5)) = 5$, $m(A_1(7)) = 7$, $m(A_1(9)) = 6$, $m(A_1(11)) = 11$, $m(A_3(2)) = 8$. By \cite{bl} $\gamma(PSL(2,q))=\gamma(PGL(2,q))=2$, so me way assume that if $n=2$ then $q$ is not a prime.
Moreover $PSL(2,4) \cong \alt(5)$, $PSL(3,2)\cong PSL(2,7),$ $PSL(2,9)\cong \alt(6)$ and $PSL(4,2)\cong \alt(8)$. In the remaining cases
inequality $(\ref{c1c8})$ holds except for $(n,q)\in \{(6,2), (5,2), (4,3), (3,3), (3,4), (3,5), (2,8), (2,16), (2,25), (2,27)\}.$  On the other hand $k^*(PSL(5,2))=13,$ $k^*(PSL(4,3))=11,$
$k(PSL(3,4))=10$, $k^*(PSL(3,5))=14,$ $k^*(PSL(2,8))=4,$
$k^*(PSL(2,25))=10$ and $k^*(PSL(2,27))=12$ (see \cite{atlas}).
Suppose $S\in\{PSL(2,16),$ $PSL(3,3)\}$: by \cite{bl} $\gamma(X)=2$ if $X=S,$ otherwise $X/S$ is a non-trivial 2-group, so every elements
in $X\setminus S$ centralizes an involution and since $X$ contains 2 conjugacy classes of  involution we deduce that $\gamma(X)\leq 3$.
Finally $\gamma(PSL(6,2))\leq [2\cdot 6 + 3\cdot d(6) + 5 + \log 6]=31 < 63/2 = m(PSL(6,2))/2.$

\item $S = B_m(q)$, $q$ odd, $m > 1$. We have $m(S) = \frac{q^{2m}-1}{q-1}$ if $q > 3$, $m(S) = \frac{1}{2} 3^m(3^m-1)$ if $q=3$ and $m>2,$ $m(B_2(3))=27.$
Moreover $k(S)\leq  7.3 \cdot q^m$ and $k(B_2(q))\leq q^2+12q$ if $q$ is odd. This is enough to deduce that $k(S)<m(S)/2,$  except in the three cases $B_2(3)$, $B_2(5), B_3(3).$ However, it follows from \cite{atlas}, that $k^*(B_2(3))=7<m(B_2(3))/2=27/2,$ $k(B_2(5))=34<m(B_2(3))/2=78$
and $k(B_3(3))=58<m(B_3(3))/2=351/2.$

\item $S = C_m(q)$, $m > 2$. We have $m(S) = \frac{q^{2m}-1}{q-1}$ if $q > 2$, $m(S) = 2^{m-1}(2^m-1)$ if $q=2$.
Moreover $k(S)\leq  15.2 \cdot q^m$. It is easy to see that $15.2 \cdot q^m < m(S)/2,$ except  for $(m,q)\in \{(3,2), (3,3), (3,4), (3,5), (4,2), (4,3), (5,2)\}.$
On the other hand by \cite[Table 3]{chev} $k(C_3(4))\leq 4^3+5\cdot 4^2< m(C_3(4))/2=1365/2$,
$k(C_3(5))\leq 5^3+12\cdot 5^2< m(C_3(5))/2=1953$
 and it follows from \cite{atlas} that $k^*(C_3(2))=9<m(C_3(2))/2=14,$
$k(C_3(3))=74<m(C_3(3))/2=182,$ $k^*(C_4(2))=15<m(C_4(2))/2=60,$ $k(C_4(3))=278<m(C_4(3))/2=1640$ and $k(C_5(2))=198<m(C_5(2))/2=248.$

\item $S = D_m(q)$, $m > 3$. We have that $m(S) = \frac{(q^m-1)(q^{m-1}+1)}{q-1}$ if $q > 2$ and $m(S) = 2^{m-1} (2^m-1)$ if $q=2$.
Moreover $k(S)\leq  6.8 \cdot q^m$ and it is easy to see that $6.8 \cdot q^m < m(S)/2$ except  for $(m,q)\in \{(4,2),(4,3)\}.$
On the other hand it follows from \cite{atlas} that $k(D_4(2))=53< m(D_4(2))/2=60$ and $k(D_4(3))=114 < m(D_4(3))/2=520.$

\item $S = \prescript{2}{}\!{A}_m(q)$, $m > 1$. We have $m(S) = \frac{(q^{m+1}-(-1)^{m+1}) (q^{m}-(-1)^{m})}{q^2-1}$ if $m \geq 4$ and $m+1$ is not divisible by $6$ when $q=2$, $m(S) = 2^{m}(2^{m+1}-1)/3$ if $q=2$ and $m$ is divisible by $6$, $m(\prescript{2}\!{}{A}_3(q)) = (q+1)(q^3+1)$, $m(\prescript{2}\!{}{A}_2(q)) = q^3+1$ if $q \neq 2, 5$, $m(\prescript{2}\!{}{A}_2(5)) = 50$. Moreover $k(S)\leq 8.26 \cdot q^m$
    and $k(\prescript{2}{}\!{A}_m(q))\leq q^{n-1}+7q^{n-2}$ if $q>2.$  This is enough to deduce that $k(S)<m(S)/2,$  except
    \emph{}when $m=2$ and $q\leq 7$, $m=3$ and $q\leq 5$ or $(m,q)=(4,2).$ However
    $k({\prescript{2}{}\!{A}}_2(3))=14\leq m(\prescript{2}{}\!{A}_2(3))/2=14,$
    $k({\prescript{2}{}\!{A}}_2(4))=22<m(\prescript{2}{}\!{A}_2(4))/2=65/2,$
    $k({\prescript{2}{}\!{A}}_2(5))=14<m(\prescript{2}{}\!{A}_2(5))/2=25,$
    $k({\prescript{2}{}\!{A}}_2(7))=58<m(\prescript{2}{}\!{A}_2(7))=172,$
    $k^*({\prescript{2}{}\!{A}}_3(2))=7<m(\prescript{2}{}\!{A}_3(2))/2=27/2,$
    $k({\prescript{2}{}\!{A}}_3(3))=20<m(\prescript{2}{}\!{A}_3(3))/2=66,$
    $k({\prescript{2}{}\!{A}}_3(4))=94<m(\prescript{2}{}\!{A}_3(4))/2=325/2,$
    $k({\prescript{2}{}\!{A}}_3(5))=97<m(\prescript{2}{}\!{A}_3(5))/2=378$ and
    $k({\prescript{2}{}\!{A}}_4(2))=47<m(\prescript{2}{}\!{A}_4(2))/2=165/2.$

\item $S = \prescript{2}{}{D}_m(q)$, $m > 3$. We have $m(S) = \frac{(q^m+1)(q^{m-1}-1)}{q-1}.$
Moreover $k(S)\leq 6.8 \cdot q^m< m(S)$ except when $(m,q)\in\{(4,2),(4,3),(5,2)\}$. Moreover it follows from \cite{atlas} that
$k(\prescript{2}{}{D}_4(2))=39 < m(\prescript{2}{}{D}_4(2))/2=119/2,$   $k(\prescript{2}{}{D}_4(3))=114 < m(\prescript{2}{}{D}_4(3))/2=533$
and $k(\prescript{2}{}{D}_5(2))=115 < m(\prescript{2}{}{D}_5(2))/2=495/2.$
\end{itemize}

Now suppose that $S$ is a Lie group of exceptional type. The bound $k(S) \leq 15.2 q^r$ (where $r$ is the rank) in \cite[Corollary 1.2]{chev} compared with \cite[Table 1]{lucdam} implies that $\gamma(X) < m(S)/2$ if $S$ is one of the groups $F_4(q)$, $^2F_4(q)$, $E_6(q)$, $^2E_6(q)$, $\prescript{3}{}{D}_4(q)$, $E_7(q)$, $E_8(q)$. Suppose this is not the case. We will use \cite[Table 1]{chev} and \cite[Table 1]{lucdam}.

\begin{itemize}

\item $S = G_2(q)$.    We have $k(S)\leq q^2+2q+9$ and $q^2+2q+9\leq 3q^2 < q^5/2 \leq m(S)/2$.

\item $S = \prescript{2}{}{G}_2(q)$, $p = 3$, $f = 2m+1$, $m \geq 1$.
We have $k(S)\leq q+8$ and $q+8 \leq (q^3+1)/2=m(S)/2$, since $q\geq 27.$

\item $S = \prescript{2}{}{B}_2(q)$, $p = 2$, $f = 2m+1$, $m \geq 1$. We have $k(S)\leq q+3$ and $q+3 < m(S)/2=(q^2+1)/2$,  since $q \geq 8$.

\item $S = \prescript{2}{}{F}_4(2)'$. In this case $k(S) = 22 < 2^5 \cdot 5^2 = m(S)/2$.
\end{itemize}
This concludes our proof.
\end{proof}



\begin{pro} \label{ell}
Let $G$ be a group with a unique minimal normal subgroup $N$ and assume that $N$ is nonabelian and $G/N$ is cyclic.
Let $N\cong S^t$ with $S$ a nonabelian simple group.
Then $\gamma(G) < t\cdot m(S)/2.$ 
\end{pro}

\begin{proof}
By assumption, $N= S_1 \times \ldots \times S_t$, with $S_i\cong S$ for
$i=1, \ldots , t$. Let  $\psi$ be the map from  $N_G(S_1) $ to  $\aut(S)$ induced by
the conjugacy  action on $S_1.$ Set $X=\psi(N_G(S_1))$ and note that $X$ is an almost simple group with socle
$S=\inn(S)=\psi(S_1)$.
Then $G$ embeds in the wreath product $X \wr \perm(t)$ \cite[Remarks 1.1.40.13]{classes}.
Since $G/N$ is cyclic, $X/S$ is also cyclic; more precisely if $h=(y_1,\dots,y_t)\rho \in G$ generates $G$ modulo $N$, then
$\rho$ is a $t$-cycle and $y_1 y_{\rho(1)} \cdots y_{\rho^{(t-1)}(1)}$ generates $X$ modulo $S.$ 

Now let $g=(x_1,\dots,x_t)\delta\in G$. If $\langle g, N\rangle \neq G,$ then $g$ is contained in one of the
 $\omega(|G/N|) = \omega(t \cdot |X/S|)$ normal subgroups of prime index containing $N.$
Assume now that $\langle g, N\rangle = G$  
and let $y:=x_1 x_{\delta(1)} \cdots x_{\delta^{(t-1)}(1)}.$
Since $\langle y, S \rangle =X$, there exists a proper subgroup
$M$ of $X$ with $y\in M$ and $MS=X.$ Choose $a_2,\dots,a_t \in S$ such that
$$x_1 a_{\delta(1)}^{-1},\ a_{\delta(1)} x_{\delta(1)} a_{\delta^2(1)}^{-1}, \ldots, a_{\delta^{t-2}(1)} x_{\delta^{t-2}(1)} a_{\delta^{t-1}(1)}^{-1},\  a_{\delta^{t-1}(1)} x_{\delta^{t-1}(1)} \in M.$$
It can be easily checked that $g$ normalizes $M \times M^{a_2} \times \ldots \times M^{a_t}.$
In other words, if $\mathcal M$ is a normal cover of $X$, then a normal cover of $G$ can be obtained
taking the maximal normal subgroups of $G$ containing $N$
and the conjugates of the normalizers $N_G(M\times\dots\times M)$ with  $M$ running in $\mathcal M.$
It follows that $$\gamma(G) \leq \omega(t \cdot |X/S|) + \gamma(X).$$
If $t=1$ then $G=X$ and the result follows from Proposition \ref{as}.
If $t \geq 2$ then,
since $\omega(|X/S|) < m(S)/4$ \cite[Lemma 2.7]{asch-gur},
$\gamma(X) \leq m(S)/2$ and $4\omega(t)\leq 4(2t-3)<m(S)(2t-3)$, we conclude
$$\gamma(G) \leq \omega(t) + \frac{m(S)}4+\frac{m(S)}{2} < \frac{t\cdot m(S)}{2}$$
as in our claim.
\end{proof}

\begin{proof}[Proof of Theorem \ref{permut}] The proof is by induction on the degree $n.$ If $G/G^\prime$ is not cyclic then $C_p\times C_p$ is an epimorphic image of $G$ for some prime $p$.
Since $G \leq \sym(n)$, $p^2$ divides $n!$ so $p \leq n/2$, and we deduce that $\gamma(G)\leq \gamma(C_p\times C_p)=p+1 \leq (n+2)/2$. So from now on we will assume that $G/G^\prime$ is cyclic. If $G$ is soluble,
then $\gamma(G)=2,$ by Lemma \ref{solva}. So we may assume that $G$ is not soluble.
First suppose that $G$ is not transitive;
let $\Omega_1,\dots,\Omega_t$ be the orbits of $G$ on $\{1,\dots,n\}$ and $G_1,\dots,G_t$ the corresponding transitive constituents.
Since $G$ is not soluble and it is a subdirect product of $G_1\times \dots \times G_t$, there exists $i$ such that $G_i$ is noncyclic:
by induction $\gamma(G)\leq \gamma(G_i) \leq (|\Omega_i|+2)/2 \leq (n+2)/2.$ So we may assume that $G$ is transitive.
Suppose
that $\{ B_1,\dots,B_s\}$ is a system of blocks for $G$ with $|B_i|=r$.
Consider $\mathop{\rm {St}}_G(B_1)$, the stabilizer
in $G$ of the block $B_1$. Denote by $\alpha:
\mathop{\rm{St}}_G(B_1) \to \perm (r)$ the permutation
representation induced by the action of $\mathop{\rm{St}}_G(B_1)
$ on the set $B_1$ and by
$\beta: G \to \perm (s)$ the permutation representation induced by
the action of $G$ on the set of blocks and let
$H =\alpha (\mathop{\rm{St}}_G(B_1))$, $K=\beta (G).$
We may identify
$G$, as a permutation group, with a subgroup of $H \wr K$ (in its
imprimitive representation)
in such a way that, for  $1 \leq i \leq s$, $\mathop{\rm{St}}_G(B_j)$
acts on $B_j$ as the subgroup $H$ of $\perm (r)$
and $G$ acts on the set $\{B_1,\dots,B_s\}$ as the subgroup $K$ of
$\perm (s)$. We choose $B_1=\{1,\dots,n\}$ if $G$ is primitive, $B_1$ to be an
imprimitive block of minimal size otherwise. If $K$ is noncyclic, then, by induction, $\gamma(G)\leq \gamma(K) \leq (s+2)/2 \leq
(n+2)/2,$ so
we may assume that $K$ is cyclic. 
We distinguish three different possibilities:

\noindent 1) $H$ has a unique minimal normal subgroup $N$ and $N\cong C_p^t$ is an elementary abelian $p$-group. In this case
$r=p^t$ and $H/N\leq GL(t,p)\leq \perm(r-1).$
Consider the normal subgroup $M\cong N^{s}$ of $H\wr \perm(s).$ Notice that $G/(M\cap G)\leq \GL(t,p)\wr \perm(s)$
has a faithful permutational representation of degree $(r-1)s$. Since $G$ is not soluble,
 $G/(M\cap G)$ is not cyclic and therefore by induction $\gamma(G)\leq \gamma(G/(M\cap G)) \leq
 ((r-1)s+2)/2\} \leq (n+2)/2.$

\noindent 2) $H$ has a unique minimal normal subgroup $N$ and $N\cong S^t$ is the direct product of $t$ isomorphic non abelian simple groups.
In particular $N$ is transitive of degree $r$ so
$r\geq m(S)^t$ (see \cite[Proposition 5.2.7]{kl} and the comment afterwards) and $G\leq H\wr \perm(s)\leq (\aut(S) \wr \perm(t))\wr \perm(s)\leq \aut(S)\wr \perm(t\cdot s).$
Consider the normal subgroup $M\cong S^{t\cdot s}$ of $\aut(S)\wr \perm(t\cdot s).$ Notice that $G/(M\cap G)\leq \out(S)\wr \perm(t\cdot s)$
has a faithful permutational representation of degree $|\out(S)|\cdot t \cdot s\leq (2\cdot m(S) \cdot t \cdot s)/3 < m(S)^ts\leq r\cdot s \leq n$
(indeed $|\out(S)|\leq 2m(S)/3$ by \cite[Lemma 2.7]{asch-gur}).
If $G/(M\cap G)$ is not cyclic, then by induction $\gamma(G)\leq \gamma(G/(M\cap G)) \leq (n+2)/2.$ Assume that $G/(M\cap G)$ is cyclic and let
$T$ be a minimal normal subgroup of $G$ contained in $M\cap G.$ We have $T\cong S^u$ with $u\leq t\cdot s$; moreover $G/C_G(T)$ has a unique
minimal normal subgroup $T^*/C_G(T)\cong T$ and $G/T^*$ is cyclic: by Proposition \ref{ell}
$\gamma(G)\leq \gamma(G/C_G(T))< u\cdot m(S)/2 \leq
t \cdot s \cdot m(S)\leq n/2.$

\noindent 3) $\soc H=N=N_1\times N_2$ where $N_1$ and $N_2$ are isomorphic non abelian minimal normal subgroups of $H$. In this case $r=|N_1|=|N_2|.$
Let $H^*:=H/C_H(N_1)\leq \aut(N_1)\leq \perm(r-1).$ We have $G\leq H\wr \perm(s)$ and the wreath product $H\wr \perm(s)$ contains a normal subgroup
$M\cong C_H(N_1)^s.$ Notice that $G/(M\cap G)\leq H^* \wr \perm(s)$ is a noncyclic permutation group of degree $(r-1)s < n.$ So
$\gamma(G)\leq \gamma(G/(M\cap G)) \leq (n+2)/2.$
\end{proof}




\begin{proof}[Proof of Proposition \ref{basso}]
Let $m=\mu(G).$  First assume that $G$ contains a maximal subgroup $M$ of index $m$, which is not normal in $G$. In this case
$G/M_G$ is a non cyclic permutation group of degree $m$ and $\gamma(G)\leq \gamma(G/M_G)\leq m$ by Theorem
\ref{permut}.
Otherwise 
 $G$ contains two normal maximal subgroup of index $m.$ In this case $m$ is a prime
and $C_m\times C_m$ is an epimorphic image of $G$. In particular $\gamma(G)\leq \gamma(C_m \times C_m)=m+1.$

Therefore we have proved that $\gamma(G)\leq m+1$ and $\gamma(G)=m+1$ only if $m$ is a prime, $C_m\times C_m$ is an epimorphic image
of $G$ and $\gamma(G)=\gamma(C_m\times C_m)=\gamma(G/G^\prime).$
\end{proof}

\section{Groups $G$ with $\gamma(G)=2$}\label{gamma2}
Before to stat our discussion, let us introduce a couple of easy observations.

\begin{lem}\label{uno}Let $H$ be a proper subgroup of a finite group $G$ and let $N\unlhd G$ be such that $HN=G.$
We have $\bigcup_{g \in G}(H\cap N)^g \neq N.$
\end{lem}

\begin{proof}Let $X=H\cap N.$ Since $X\unlhd H,$ we have $G=HN=N_G(X)N.$ Hence $\cup_{g\in G} X^g=\cup_{n \in N}X^n \neq N.$
\end{proof}

\begin{lem}\label{due}Let $H$ be a proper subgroup of a finite group $G$ and let $N_1, N_2$ be two different minimal normal subgroups
of $G.$ If $HN_1=HN_2=G,$ then $H\cap N_1=H\cap N_2=1.$
\end{lem}

\begin{proof}Assume $HN_1=HN_2=G.$ Then $H\cap N_1$ is normalized by $H$ and centralized by $N_2$ hence $H\cap N_1$ is normalized
by $HN_2=G$. Since $N_1$ is a minimal normal subgroup of $G$ and $N_1\not\leq H$, we must have $H\cap N_1=1.$
\end{proof}

For the remaining part of this section, $G$ will be a finite group with the following properties:
\begin{enumerate}
\item $\gamma(G)=2;$
\item $\gamma(G/N)>2$ if $N$ is a non trivial normal subgroup of $G.$
\end{enumerate}
In particular there exists two maximal subgroups $H$ and $K$ with $$G=(\bigcup_{x \in G}H^x)\bigcup (\bigcup_{y \in G}K^y).$$
Moreover $(H\cap K)_G=1$, otherwise we would have $\gamma(G/(H\cap K)_G)=2.$
Let $$M=\soc (G)=N_1\times \dots\times N_t$$ be the socle of $G$ with $N_i$ a minimal normal subgroup of $G$ for $1\leq i \leq t.$

\begin{lem}$t=1$ i.e. $G$ contains a unique minimal normal subgroup.
\end{lem}

\begin{proof} We distinguish 2 cases:

\noindent a) One of the two subgroups $H$ and $K$ contains $M$.

\noindent Assume for example $M\leq H.$ In this case $K_G\cap M \leq (K\cap H)_G=1,$
hence $K_G=1$ and $t$ is the number of minimal normal subgroups of a primitive permutation group $G$ with point stabilizer $K.$ Assume by contradiction that $t\neq 1.$ Then $t=2$ and (see for example \cite[Proposition 1.1.12]{classes}) we may assume that there exists a monolithic primitive group $L$ with non abelian socle $N$ and a subgroup $T$ of $L$ with $N\leq T<L$ such that
$$G=\{(l_1,l_2)\in L^2 \mid Nl_1=Nl_2\},\quad M=N^2,\quad K=\{(l,l)\mid l \in L\}, \quad H=T^2\cap G.$$
Let $x\notin \cup_{l\in L}T^l$ and consider the coset $\Omega=(x,x)N^2.$ Clearly $\Omega \cap (\cup_{g\in G}H^g)=\emptyset$ hence
$\Omega \subseteq \cup_{g\in G}K^g$. Let $R=\{(1,n)\mid n \in N\}\subseteq G.$ Since $KR=G,$ we have
$\cup_{g\in G}K^g=\cup_{r\in R}K^r$, hence $$\{(x,xn)\mid n \in N\} \subseteq \Omega \subseteq \{(l,l^n)\mid l\in L, n\in N\}.$$
In particular $Nx=\{x^n\mid n\in N\}$ and this implies $C_N(x)=1,$ i.e. $N$ admits a fixed-point free automorphism: by \cite{fixed-point-free}
$N$ is a soluble group, a contradiction.


\noindent b) $HM=KM=G.$
\noindent Let us define the following two subsets of $\Omega=\{1,\dots,t\}$:
$$\Omega_H=\{i \in \Omega \mid N_i\cap H=1\}, \quad \Omega_K=\{i \in \Omega \mid  N_i\cap K=1\}.$$
We claim that $\Omega_H=\Omega_K=\emptyset.$
To prove this, assume for example that $\Omega_H=\{1,\dots,u\}$ with $u\neq 0.$ By Lemma \ref{due}, $N_i \leq H$ for all $i>u.$
Moreover if $i\leq u$, then $N_i\cap H^g=(N_i\cap H)^g=1,$ hence $N_i \leq \cup_{g\in G} K^g.$ It follows that $N_i=\cup_{g\in G} N_i\cap K^g=
\cup_{g \in G}(N_i\cap K)^g$: by Lemma \ref{uno} we must have $N_i \leq K.$ Since $KM=G$, there exists $j$ such
that $KN_j=G.$ We have $j>u$ hence $N_j\leq H.$
By Lemma \ref{uno} there exists $x \in N_j \setminus (\cup_{g\in G} N_j\cap K^g).$ Take $1\neq y \in N_1$ and consider $z=yx.$ We cannot have $z \in H^g$
(since $x \in N_j=N_j^g\subseteq H^g$, $z \in H^g$ would imply $y\in H^g \cap N_1=1$). Hence $z=yx \in K^g$ for some $g$, however $y \in N_1\leq K^g$
hence $x \in K^g,$ a contradiction. So our claim that $\Omega_H=\Omega_K=\emptyset$ has been proved. Combined with
Lemma \ref{due} and the fact that $(H\cap K)_G=1,$  this implies that if $t\neq 1$ then $t=2$ and we may assume
$N_1\leq H,$ $N_2\leq K$ and $N_2H=N_1K=G.$  By Lemma \ref{uno}, there exist $x\in N_1 \setminus\cup_{g\in G} K^g$ and $y \in N_2\setminus \cup_{g\in G} H^g.$
Consider $z=xy.$ If $z \in H^g,$ then since $x \in N_1=N_1^g \leq H^g$ we would have $y \in H^g,$ a contradiction. Similarly, we cannot have $z\in K^g.$ This proves that $t=1.$
\end{proof}

\begin{lem}If neither $H$ nor $K$ contains $\soc(G)$, then $G$ is an almost simple group.
\end{lem}

\begin{proof}Let $M=\soc (G)$ and assume $G=HM=KM.$  We have
$M\subseteq (\cup_{x \in M}H^x)\cup (\cup_{y \in M}K^y)$ and this implies
$$
M= (\cup_{x \in M}(H\cap M)^x)\cup (\cup_{y \in M}(K\cap M)^y). \quad \quad (*)
$$
Together with Lemma \ref{uno}, this implies $H\cap M\neq 1$ and $K\cap M\neq 1.$
In particular, if $M$ is abelian, then $M\leq H\cap K,$ a contradiction.
Therefore $M$ is a direct product of $r$ copies of a non-abelian simple group $S.$
Assume, by contradiction, that $r\neq 1.$ To fix the notation, let
$M= S_1 \times \cdots \times S_r$
 and $\pi: M \mapsto S$ the map  induced by  the projection of $M$  on the first component.
The maximal  subgroups $X$ of $G$ with $XM=G$ and $X\cap M\neq 1$ are of one of the following types:
\begin{enumerate}
\item[a)]product type: if $1<\pi (M \cap X) <S$;
\item[b)]diagonal type: if $\pi (M \cap X) =S$.
\end{enumerate}
In the first case $X\cap M\cong T_1\times \dots \times T_r$ with $1<T_i<S_i$ and $T_i\cong T_j$ for every $1\leq i\leq j \leq r.$
In the second case there exists a partition $\Phi$ of $\{1,\dots,r\}$
such that
$X \cap M=\prod_{B \in \Phi}D_B$,
where all the blocks have the same cardinality and, for every block $B \in \Phi$, $|B|\neq 1$ (otherwise we would have
 $X \cap M=M$ hence $X=G$) and $D_B$ is a full diagonal subgroup of
$\prod_{j \in B}S_j$ (that is, if $B=\{j_1,\dots j_t\}$, there exist $\phi_2,
\dots, \phi_t \in \aut S$ such that $D_B=\{(x,x^{\phi_2},\dots,x^{\phi_t})\mid x
\in S\} \le S_{j_1} \times \cdots \times S_{j_t}$).
We have three possibilities:
\begin{enumerate}
\item $H$ and $K$ are both of diagonal type. Let $\Delta=\{(s,1,\dots,1)\mid s\in S, s\neq 1\}\subseteq M.$
By the way in which maximal subgroups of diagonal type are defined, $\Delta\cap H^m=\Delta\cap K^m=\emptyset$ for each $m\in M,$
against $(*).$
\item $H$ is of product type and $K$ is of diagonal type. We have $H\cap M=T_1 \times \dots \times T_r$ with $T=T_1<S.$
There exists $s\in S\setminus \cup_{s\in S}T^s$. Consider $m=(s,1,\dots,1):$ $m\notin (\cup_{x \in M}(H\cap M)^x)\cup (\cup_{y \in M}(K\cap M)^y),$ against  $(*).$
\item $H$ and $K$ are both of product type. Let $H\cap M=T_1 \times \dots \times T_r$ and $K\cap M=U_1 \times \dots \times U_r.$
Since $T_1$ and $U_2$ are proper subgroup of $S,$ there exist $a\in S\setminus \cup_{s\in S}T^s$
and $b\in S\setminus \cup_{s\in S}U_2^s$. Consider $m=(a,b,1,\dots,1):$ $m\notin (\cup_{x \in M}(H\cap M)^x)\cup (\cup_{y \in M}(K\cap M)^y),$
against  $(*).$
\end{enumerate}
All the possibilities lead to a contradiction, hence it must be $r=1$ and $G$ is an almost simple group.
\end{proof}

We recall some results concerning almost simple groups $G$ with $\gamma(G)=2.$
It was shown by H. Dye \cite{dy} that the symplectic group $G=\text{Sp}_{2l}(2^f)$ defined over a finite field of characteristic 2 is the union of the two $G$-conjugacy classes of subgroups isomorphic to $O^+_{2l}(2^f)$ and $O^-_{2l}(2^f)$ embedded naturally. D. Bubboloni, M.S. Lucido and T. Weigel \cite{blw} notices the existence of an interesting example in characteristic 3, i.e. in $G=F_4(3^f)$ every element is conjugated to an element of the subgroup $B_4(3^f)$ or of the subgroup $3.{^3D_4(3^f)}.$ In \cite{bbh} it is proved that $\gamma(\alt(n))=2$ if and only if $4 \leq n \leq  8$, $\gamma(\perm(n))=2$ if and only if $3\leq n \leq 6.$
In \cite{bl} it is proved that $\gamma(PSL(n,q))=\gamma(PGL(n,q))=2$ if and only if $2 \leq n \leq 4.$
Another example is given by the Mathieu group $\M_{11}$ \cite[Claim 5.1]{attila}.

\

In the remaining part of the section we concentrate our attention in the case when $\soc (G)\leq H$
(and consequently $G=KM$).

\begin{lem}\label{prev}Assume that $H$ and $K$ are maximal subgroups of a primitive monolithic group $G$ with $M=\soc(G)\leq H$ and $KM=G.$ Let $R=K\cap M.$
The following are equivalent:
\begin{enumerate}
\item $G=(\cup_{x \in G}H^x)\bigcup (\cup_{y \in G}K^y);$
\item $gM \subseteq \cup_{m \in M}K^m$ for each  $g \in G \setminus \cup_{x \in K}H^x;$
\item $gM=\cup_{m \in M}(gR)^m$ for each $g \in K \setminus \cup_{x \in G}H^x;$
\item whenever $g\in K\setminus \cup_{x \in K}H^x$ and $m\in M,$ we have $m\in R$ if and only if $(gR)^m=gR.$
\end{enumerate}
\end{lem}
\begin{proof}Since $KM=G$ and $M\leq H$, we have $\Gamma=\cup_{x \in G}H^x=\cup_{x \in K}H^x$. Moreover $gM\cap \Gamma \neq \emptyset$ if and only if $gM\subseteq \Gamma$.
Equivalently, if $g \notin \Gamma,$ then $gM\cap \Gamma=\emptyset.$ It follows that (1) holds if and only if
$gM \subseteq \cup_{x \in G}K^x
=\cup_{m \in M}K^m \text{ whenever } g \notin \Gamma$, i.e. (1) and (2) are equivalent.
Assume that $(2)$ holds. In particular if  $g\in K \setminus \Gamma,$ then for each $m_1\in M,$ there exists
$m_2 \in M$ with $gm_1\in K^{m_2}$; it follows that $(gm_1)^{m_2^{-1}}=g[g,m_2^{-1}]m_1^{m_2^{-1}}\in K$ hence, since $g \in K,$
we have $[g,m_2^{-1}]m_1^{m_2^{-1}}\in K\cap M=R$ and consequently $gm_1\in (gR)^{m_2}$. Therefore $(2)$ implies (3).
Conversely, assume that (3) holds and let $g \notin \Gamma.$ Since $KM=G$ and $M\leq H$, there exists $\bar g \in K\setminus \Gamma$ with $\bar gM=gM$,
hence $gM=\bar g M = \cup_{m \in M}(\bar gR)^m\leq \cup_{m\in M}K^m.$ So (3) implies $(2)$.
Now let $a=|R|, b=|M:R|$ and let $m_1,\dots,m_b$ be a transversal of $R$ in $M.$ Notice
that if $g \in K,$ then $R$ is normalized by $g$ and  $(gR)^r \subseteq gR$ for all $r \in R.$ This implies that (3) is equivalent to
$$gM=\cup_{1\leq i \leq b}(gR)^{m_i} \text { for each } g \in K \setminus \cup_{x \in K}H^x.$$
Since $|gM|=a\cdot b$ and $|gR|=a$, the previous condition is satisfied if and only if the subsets $(gR)^{m_i}$ are pairwise disjoint;
on the other hand these subsets are disjoint if and only if the only elements $m$ of $M$ with $(gR)^m=gR$ are those of $R.$ Therefore
(3) and (4) are equivalent.
\end{proof}

Let us introduce some additional definitions. Let $M$ be an elementary abelian group and $K$ be an irreducible subgroup of
$\aut (M).$ Consider the subset $$K^*=\{k\in K \mid C_M(k)\neq 1\}.$$ We will say that $K$
is {\sl{almost-transitive}} if there exists a proper subgroup $T$ of $K$ with $K^*\subseteq \cup_{x \in K} T^x.$
If this situation holds, we have that $\gamma(M\rtimes K)=2.$ Indeed if $k\in K$ and $C_M(k)=1,$ then $kM=\{k^m \mid m \in M\}$, hence
$M\rtimes K$ can be covered by the conjugates of the two subgroups $K$ and $M\rtimes T.$

\begin{cor} If $\soc(G)=M$ is abelian, then $G=M\rtimes K$ and $K$ is an almost transitive irreducible subgroup of $\aut(M).$
\end{cor}

\begin{proof}Since $KM=G$ and $\frat(G)=1$, it must be $G=M\rtimes K$ and $M$ is an irreducible $K$-module.
Let $T=K\cap H$ and assume $g\in K\setminus \cup_{x\in K}H^x.$  Since $R=K\cap M=1$, it follows from Lemma \ref{prev} that $g^m=g$ if and only if $m=1.$ This implies $K^*=\{k\in K \mid C_M(k)\neq 1\}\leq K\cap(\cup_{x\in K}H^x)=\cup_{x \in K}(K\cap H^x)=\cup_{x \in K} T^x.$
\end{proof}

In virtue of the previous result, it should be interesting to classify the almost-transitive irreducible groups.
There are two extreme situations, one is when $K$ is an irreducible fixed point free subgroup of $\aut M$
(and consequently $G=M\rtimes K$ is a Frobenius group), the other is when $K$ is a transitive irreducible
subgroup of $\aut M$ (and consequently $G=M \rtimes K$ is a 2-transitive permutation group of degree $|M|).$
However, other possibilities occur, as the  following three examples indicate.
\begin{enumerate}

\item Let $M$ be the additive group of the finite field $F$ with 16 elements.
The multiplicative group $F^*$ contains a subgroup $Q$ of order 5, which is normalized
by the Frobenius automorphism $\sigma: f\to f^2.$ The semidirect product $K=Q\rtimes \langle \sigma
\rangle$ is an irreducible subgroup of $\Gamma L(1,16)\leq \aut(M)$; moreover $K^*=K\setminus Q$ is contained
in the union of the conjugates of a Sylow 2-subgroup. Hence $K$ is almost transitive.

\item Assume that $p > 2$ and $q$ are two prime numbers, with $p$ dividing $q-1.$ The multiplicative group $F^*$ of the field $F$ with
$q$ elements contains a cyclic subgroup $X=\langle x \rangle$ of order $p.$ Let $Y$ be the subgroup of $\perm(p)$ generated by the
permutation $\sigma=(1,2,\dots,p).$ Consider the  following subgroup $K$ of the wreath product $X\wr Y:$
$$K=\{(x^{a_1},\dots,x^{a_p})\sigma^i \in X \wr Y \mid \sum_{1\leq j\leq p} a_j=i\mod p\}.$$
The wreath product $X\wr Y$ acts on a $p$-dimensional $F$-vector space $M$ and $K$ is an irreducible subgroup of $X\wr Y \leq \aut(M).$
Suppose now  that $k=(x^{a_1},\dots,x^{a_p})\sigma^i \in K^*.$ There exists $(0,\dots,0)\neq m=(b_1,\dots,b_p)\in M$ such that
$$(b_1,\dots,b_p)=m=m^k=(b_1x^{a_1},\dots,b_px^{a_p})^{\sigma^i}.$$
Since $m\neq 0$, there exists $1\leq j \leq p$ with $b_j\neq 0.$ It must be $i=0 \mod p,$ otherwise the previous
equality  would imply that $b_r\neq 0$ for each $1\leq r\leq p$
and $b_1\cdots b_p=b_1x^{a_1}\cdots b_px^{a_p}$, and consequently $0= \sum_j a_j = i \mod p.$
It follows that $K^*\subseteq T=\{(x^{a_1},\dots,x^{a_p})\mid \sum_j a_j=0 \mod p\}.$

\item  Let $F$ be a field with $q^2$ elements, $q \equiv 3 \mod 4$, $q \neq 3$, and consider the 2-dimensional vector space $M=F^2.$ The multiplicative group $F^*$ contains two cyclic subgroups $A$ and $B$ of orders, respectively,
$(q-1)/2$ and $q+1$.
The Frobenius automorphism $\sigma: f\to f^q$ normalizes $B$ and centralizes $A.$
Consider the subgroup $K$ of $\Gamma L(1,q^2)\wr \langle (1,2)\rangle \leq \aut(M)$
defined as follows:
$$K=\{(1,2)^r\sigma^t(ab_1,ab_2)\mid a\in A, b_1,b_2 \in B, 0\leq r,t \leq 1\}.$$
Assume $k=(1,2)^r\sigma^t(ab_1,ab_2)$ has a non trivial fixed point $(f_1,f_2)\neq (0,0)$. There are two possibilities:
\begin{enumerate}
\item If $r=0$ we have $(f_1,f_2)=(f_1,f_2)^k=(f_1^{q^t}ab_1,f_2^{q^t}ab_2).$ There exists $i\in\{1,2\}$ with $f_i\neq 0$ and
we must have that $f_i^{q^t}ab_i=f_i,$ i.e. $a=f_i^{1-q^t}b_i^{-1}.$ Since $f_i^{1-q^t}\in B,$ we conclude that
$a \in A\cap B=1.$
\item If $r=1$ we have $(f_1,f_2)=(f_1,f_2)^k=(f_2^{q^t}ab_1,f_1^{q^t}ab_2).$ We must then have $f_1\neq 0,$ $f_2\neq 0,$
$f_2=f_1^{q^t}ab_2,$ $f_1=f_2^{q^t}ab_1=(f_1^{q^t}ab_2)^{q^t}ab_1=f_1a^{q^t+1}b_2^{q^t}b_1=f_1a^2b_2^{q^t}b_1,$
hence $a^2=(b_2^{q^t}b_1)^{-1}\in A\cap B=1;$ since $|A|=(q-1)/2$ is odd, we conclude that $a=1.$
\end{enumerate}
It follows that $K^* \subseteq T=\{(1,2)^r\sigma^t(b_1,b_2)\mid  b_1,b_2 \in B, 0\leq r,t \leq 1\}.$
\end{enumerate}

\

We conclude this section with discussing some  examples in which $M=\soc (G)$ is nonabelian, $M\leq H$ and $MK=G.$

\

Let $S$ be a finite non abelian simple group and let $p$ be a prime which does not divide $|S|.$ Consider the wreath product
$G=S \wr \langle \sigma \rangle$ with $\sigma=(1,2,\dots,p)\in \perm(p).$ We claim that $\gamma(G)=2.$ More precisely let $M=S^p$ be the
base of the wreath product and let $H=\{(s,\dots,s)\sigma^i\mid s \in S,\ 0 \leq i \leq p-1\}$ be a maximal subgroup of $G$ of diagonal type. We prove that
$G=M\cup (\cup_{m \in M}H^m).$ Indeed consider for example $(t_1,\dots,t_p)\sigma \in G$. We look for $s,x_1,\dots,x_p\in S$ such that
$$(t_1,\dots,t_p)\sigma=((s,\dots,s)\sigma)^{(x_1,\dots,x_p)}=(x_1^{-1}sx_2,x_2^{-1}sx_3,\dots,x_p^{-1}sx_1)\sigma.$$
We can take
$$\begin{aligned}
x_1&=1 \\x_2&=s^{-1}t_1\\ x_3&=s^{-2}t_1t_2\\&\dots \dots \dots \\ x_p&=s^{-(p-1)}t_1t_2\cdots t_{p-1}\\s^p&=t_1t_2\cdots t_p
\end{aligned}$$
where the existence of $s$ is ensured from the fact that $p$ does not divide $|S|.$

\

We want to discuss the existence of examples in which $M=\soc (G)=S^n$, with $S$ a nonabelian simple group, $M\leq H$ and
$K$ is a maximal subgroup of $G$ of product type.  We have $M=S^n\leq G \leq \aut(S)\wr \perm(n)$ and it is not restrictive to assume that
$R=K\cap M=T^n$ with $T < S.$ There exists $g \in K \setminus \cup_{x\in G}H^x;$  we can write $g$ in the form $g=(h_1,\dots,h_n)\sigma$ with $\sigma \in \perm(n)$ and $h_i \in \aut S.$
Since $g$ normalizes $R,$ we have that $h_i$ normalizes $T$ for each $1\leq i\leq n.$ Let $\Omega\subseteq \{1,\dots,n\}$ be the $\sigma$-orbit containing 1.
It is not restrictive to assume that $\Omega=\{1,\dots,r\}$ and $\sigma=\rho \tau$ where $\rho=(1,2,\dots,r)$ and $\tau$ fixes pointwise the elements
of $\Omega$ (we don't exclude the possibility $r=1).$ Let $U=S^r,$ $V=T^r$ and let $y=(h_1,\dots,h_r)\rho \in \aut(S)\wr\perm(r).$ By Lemma \ref{prev} (3), we must have
$$yU=\cup_{u \in U}(yV)^u.$$ Recall that if $u=(y_1,\dots,y_r)\in U$ then
$$y^u=(y_1,\dots,y_r)^{-1}(h_1,\dots,h_r)\rho (y_1,\dots,y_r)=(y_1^{-1}h_1y_2,y_2^{-1}h_2y_3,\dots,y_r^{-1}h_ry_1)\rho.$$
In particular, given $s \in S,$ there exist $x_1,\dots,x_r \in T$ and $y_1,\dots,y_r \in S$ such that
$$(h_1,\dots,h_rs)\rho=(y_1^{-1}h_1x_1y_2,y_2^{-1}h_2x_2y_3,\dots,y_r^{-1}h_rx_ry_1)\rho$$ and this implies
$$h_1\cdots h_rs=(y_1^{-1}h_1x_1y_2)(y_2^{-1}h_2x_2y_3)\cdots (y_r^{-1}h_rx_ry_1)= y_1^{-1}h_1x_1\cdots h_rx_ry_1.$$
But then, setting $h=h_1\cdots h_r \in \aut(S)$ we must have
$$hS=\cup_{s\in S}(hT)^s\quad (**).$$
The previous equality cannot occur if $h\in S;$ otherwise we would have
$hS=S=\cup_{s\in S}(\langle h\rangle T)^s,$ which implies $S=\langle h \rangle T,$ and consequently, since $h\in N_S(T),$ $T\unlhd S.$
For some choices of $S$, it is impossible to find $h\in \aut(S)\setminus S$ and $T<S$ satisfying $(**).$ Assume for example $S=\alt(n),$
with $n\neq 6.$ If $(**)$ holds, since $h\notin S=\alt(n)$ we would have $h\alt(n)=(1,2)\alt(n)\subseteq\cup_{s\in S}(\langle h \rangle T)^s.$
In particular $\langle h \rangle T$ would be a proper subgroup of $\perm(n)$ containing at least one conjugate of every odd permutation.
The situation is different for $S=\alt(6).$ In this case consider $G=M_{10} \leq \aut (S).$
$G\setminus S$ consists of three conjugacy classes whose representatives have orders respectively 4, 8, 8. So
$G \setminus S$ is covered by the Sylow 2-subgroups and $\gamma(G)=2$. But we may consider also the group $\Gamma=(S\times S)\langle \gamma\rangle$
with $\gamma=(g,1)\epsilon$, where $\epsilon=(1,2)$ and $g\in M_{10}\setminus S.$ This group $\Gamma$ contains a normal subgroup $M=S^2$ of index 4: we claim that if $x\in \Gamma
\setminus M$ then $|x|$ divides 16. Indeed one of the following holds:
\begin{enumerate}
\item $x=(gs_1,s_2)\epsilon$ for some $s \in S.$ Then $x^2=(gs_1s_2,s_2gs_1)$ has either order 4 or 8.
\item $x=(s_1,gs_2)\epsilon$ for some $s \in S.$ Then $x^2=(s_1g_2s_2,gs_2s_1)$ has either order 4 or 8.
\item $x=(gs_1,gs_2)$ for $s_1,s_2 \in S.$ Then $|x|$ divides 8.
\end{enumerate}
But then any element of $\Gamma$ belongs either to $M$ or to a Sylow 2-subgroup, hence $\gamma(\Gamma)=2.$

\

A more general family of examples can be obtained in the following way. Let $S=\ssl(2,2^p)$ with $p\geq 5$ a prime and let $A=\aut S=S\langle \phi \rangle$ with $\phi$ the Frobenius automorphism. Since $p\neq 3$ we have that $(|S|,p)=1$. In particular if $a\in A \setminus S,$ then $|a|$
is divisible by $p$ hence $a$ centralizes a Sylow $p$-subgroup of $A.$
This implies that $A\setminus S \subseteq \cup_{s \in S}H^s$ where $H=C_A(P)$ and $P$ is a Sylow $p$-subgroup of $A.$ Consider now the group
$G=S^p\langle x \rangle \leq A\wr \langle \sigma \rangle$ , where $\sigma=(1,2,\dots,p)$ and $x=(\phi,1,\dots,1) \sigma.$
Let $M=S^p$. Notice that $G/M$ is cyclic of order $p^2.$ In particular if $g \in G\setminus M,$ then
$p$ divides $|g|$ hence $g\in C_G(K)$ for a cyclic subgroup $K$ of order $p.$ On the other hand, the Sylow
$p$-subgroups of $G$ are cyclic of order $p^2$ and $K=\langle x^p\rangle^m$ for some $m\in M.$ This implies that $g\in H^m,$ for
$H=C_G\langle x^p\rangle.$ It follows that $\gamma(G)=2.$


\begin{thebibliography}{10}
\bibliographystyle{amsplain}
\bibitem{ash} M. Aschbacher, On the maximal subgroups of the finite classical groups. Invent. Math. 76 (1984), no. 3, 469-–514.
%
\bibitem{asch-gur} {M. Aschbacher and R. Guralnick},
On abelian quotients of primitive groups, Proc. Amer. Math. Soc., {107} (1989) {89--95}.
\bibitem{Bl} S. Blackburn,  Sets of permutations that generate the
symmetric group pairwise, J. Combin. Theory Ser. A
{113} (2006), no. 7, 1572--1581.

\bibitem{classes}
A. Ballester-Bolinches and L.~M. Ezquerro, Classes of finite
  groups, Mathematics and Its Applications (Springer), vol. 584, Springer,
  Dordrecht, 2006.

\bibitem {bbh} R. Brandl and D. Bubboloni and I. Hupp, Polynomials with roots mod $p$ for all primes $p,$ {J. Group Theory} {4} (2001) 233-–239.


\bibitem{BEGHM} J. R. Britnell, A. Evseev, R. M. Guralnick, P. E. Holmes and A.
Mar\'oti, Sets of elements that pairwise generate a linear
group, J. Combin. Theory Ser. A. {115} (2008), no.
3, 442-465.
%
\bibitem{mb} J. R. Britnell and A. Mar\'oti, Normal Coverings of Linear Groups, Algebra Number Theory, to appear
%
\bibitem{BFS} R. A Bryce, V. Fedri and L. Serena,  Subgroup coverings
of some linear groups, Bull. Austral Math. Soc.
{60}, (1999), no. 2, 227-238.
%
\bibitem{bl} D. Bubboloni and M. S. Lucido, Coverings of linear groups,  Comm. Algebra {30} (2002), no. 5, 2143–-2159.
%
%
\bibitem{blw} D. Bubboloni, M. S. Lucido and T.  Weigel, Generic 2-coverings of finite groups of Lie type,
Rend. Semin. Mat. Univ. Padova {115}, 209--252 (2006).
%
\bibitem{bp} D. Bubboloni and C. E. Praeger, Normal Coverings of Symmetric and Alternating Groups; Journal of Combinatorial Theory, Series A, 2011.
%

\bibitem{cohn}{J. H. E. Cohn}, On $n$-sum groups, Math. Scand.
{75} (1) (1994) {44--58}.
%
\bibitem{atlas}
J. H. Conway, R. T. Curtis, S. P.  Norton, R. A.  Parker \and  Wilson, R. A,
Atlas of finite groups.
Maximal subgroups and ordinary characters for simple groups. With computational assistance from J. G. Thackray,
 {Oxford University Press}, Eynsham, 1985. xxxiv+252 pp.
%
\bibitem{cl} E. Crestani and A. Lucchini, Normal coverings of soluble groups, Arch. Math.  98 (2012), no. 1, 13-–18.
%
\bibitem{lucdam} E. Damian and A. Lucchini, On the Dirichlet polynomial of finite groups of Lie type, Rend. Sem. Mat. Univ. Padova  115 (2006),
51-–69.
%
\bibitem{cubo} E. Detomi and A. Lucchini, On the Structure of Primitive $n$-Sum Groups, CUBO A Mathematical Journal 10  (2008), 195--210 .
\bibitem {dy} R. H. Dye, Interrelations of Symplectic and Orthogonal Groups in Characteristic Two, J.
Algebra {59} (1979), 202--221.
%
\bibitem{chev} J. Fulman and R. Guralnick,
Bounds on the number and sizes of conjugacy classes in finite Chevalley groups with applications to derangements,
Trans. Am. Math. Soc. 364, No. 6 (2012), , 3023--3070.
%
\bibitem{marspo} P. E. Holmes and A. Mar\'oti, Pairwise generating and covering sporadic simple groups, J. Algebra  {324}  (2010), 25--35.
 %
 \bibitem{gap}
  The GAP~Group, {GAP -- Groups, Algorithms, and Programming,
  Version 4.4.12};
  2008,
  \verb+(http://www.gap-system.org)+.
  %
 \bibitem{mg} M. Garonzi, Finite Groups that are the union of at most 25 proper subgroups, J. Algebra Appl.
Vol. 12, No. 4 (2013) 1350002
%
\bibitem{gl} M. Garonzi and A. Lucchini,
Direct products of finite groups as unions of proper subgroups,
Arch. Math.  95 (2010), no. 3, 201-–206.
%
\bibitem{gm} M. Garonzi and A. Mar{\'o}ti, Covering certain wreath products with proper subgroups,
J. Group Theory 14 (2011), no. 1, 103–-125.
%
%
\bibitem{kalu} W. M. Kantor and A. Lubotzky,  The probability of generating a finite classical group, Geom. Dedicata 36 (1990), no. 1, 67–-87.
%
\bibitem{kl} P. Kleidman and M. Liebeck, The Subgroup Structure of the Finite Classical Groups, Cambridge University Press.
%
%
\bibitem{attila} {A. Mar{\'o}ti},
Covering the symmetric groups with proper subgroups, J. Combin. Theory Ser. A {110} (2005), no. 1, {97--111}.
%
\bibitem{fixed-point-free} P. Rowley,
Finite groups admitting a fixed-point-free automorphism group,
J. Algebra {174} (1995), no. 2, 724--727.
\bibitem{sc} G. Scorza, I gruppi che possono pensarsi come somma di tre loro sottogruppi, Boll. Un. Mat.
Ital 5 (1926) 216--218.
\bibitem{tom} {M. J. Tomkinson}, Groups as the union of proper subgroups,
Math. Scand. {81} (2) (1997) 191--198.
\end{thebibliography}
\end{document}